\documentclass[final,3p]{elsarticle}

\usepackage{amsmath,amsthm,amssymb,mathrsfs}
\usepackage{enumerate}
\usepackage{slashed}
\usepackage{txfonts,dsfont}
\usepackage{aliascnt}
\usepackage[breaklinks,colorlinks]{hyperref}

\theoremstyle{plain}

\newtheorem{theorem}{Theorem}[section]

\newaliascnt{lem}{theorem}
\newtheorem{lem}[lem]{Lemma}
 \aliascntresetthe{lem}

\newaliascnt{cor}{theorem}

 \aliascntresetthe{cor}

\newaliascnt{prop}{theorem}

 \aliascntresetthe{prop}

\theoremstyle{remark}

\newtheorem{rem}{Remark}[section]

\theoremstyle{definition}

\numberwithin{equation}{section}

\newcommand{\norm}[1]{\left\lVert#1\right\rVert}
\newcommand{\abs}[1]{\left\lvert#1\right\rvert}
\newcommand{\set}[1]{\left\{#1\right\}}
\newcommand{\hin}[2]{\left\langle#1,#2\right\rangle}
\newcommand{\rin}[2]{\left(#1,#2\right)}

\newcommand{\field}[1]{\mathbb{#1}}
\newcommand{\R}{\field{R}}
\newcommand{\Com}{\field{C}}
\newcommand{\Quat}{\field{H}}
\newcommand{\Cl}{\mathbb{C}l}
\newcommand{\To}{\longrightarrow}
\newcommand{\Rmn}[1]{\uppercase\expandafter{\romannueral#1}}

\DeclareMathOperator{\vol}{vol}
\DeclareMathOperator{\area}{area}
\DeclareMathOperator{\Id}{Id}
\DeclareMathOperator{\End}{End}

\journal{arXiv}
\allowdisplaybreaks

\begin{document}

\begin{frontmatter}

\title{Extrinsic eigenvalue estimates for Dirac operator\tnoteref{CS}}

\author[whu]{Qun Chen}
\ead{qunchen@whu.edu.cn}

\author[whu]{Linlin Sun\corref{sll}}
\ead{sunll@whu.edu.cn}

\tnotetext[CS]{This research is partially supported by the National Natural Science Foundation of China (Grant Nos. 11571259, 11801420)  and Fundamental Research Funds for the Central Universities (Grant No. 2042018kf0044).}

\address[whu]{School of Mathematics and Statistics \& Computational Science Hubei Key Laboratory, Wuhan University, 430072 Wuhan, China}

\cortext[sll]{Corresponding author.}

\begin{abstract}
In this note, we prove lower and upper bounds for Dirac operators of submanifolds in certain ambient manifolds in terms of conformal and extrinsic  quantities.

\end{abstract}

\begin{keyword}
spectral theory\sep Dirac operator\sep geometry of submanifolds

 \MSC[2010] 53C27\sep 53C40

\end{keyword}

\end{frontmatter}

\section{Introduction}

The eigenvalues of Dirac operators on spin manifolds are extensively studied. In 1980, Friedrich \cite{Fri80} first derived the  lower bound of the first eigenvalues of a Dirac operator $D$ in terms of the scalar curvature $S_M$ and dimension $m$ of the underling manifold $M^m$:
 \begin{equation*}
\lambda^2\left(D\right)\geq\dfrac{m}{4(m-1)}S_M.
\end{equation*}
Since then, various kinds of estimates in terms of  intrinsic geometric
quantities have been proved (see e.g. \cite{Fri00, Gin09} and the references therein). A well known result  of Hijazi \cite{Hij86} states that
\begin{equation*}
\lambda^2\left(D\right)\geq\dfrac{m}{4(m-1)}\lambda_1(L_M)
\end{equation*}
for $m\geq3$, where $L_M=-\frac{4(m-1)}{m-2}\Delta+S_M$ is the Yamabe operator of $M$. If $m=2$, B\"ar \cite{Bar92} proved that
\begin{equation*}
\lambda^2\left(D\right)\geq\dfrac{4\pi(1-g_M)}{\area(M)},
\end{equation*}
where $g_M$ is the genus of $M$.

\par
On the other hand, the submanifold theory for Dirac operators was introduced by B\"ar in \cite{Bar98}.
Let $M^m\stackrel{\iota}{\hookrightarrow}\bar M^{m+n}$ be a closed oriented connected spin submanifold isometrically immersed in a Riemannian spin manifold $\bar M^{m+n}$ with fixed spin structures.  Milnor's Lemma claims that there is a unique spin structure \cite{LawMic89} on the normal bundle $N$ of $M$ in $\bar M$.  Denoted by $\bar\nabla, \nabla$ and $\nabla^{\perp}$ the Levi-Civita connections on $\bar M, M, N$ respectively. Denoted by $\nabla^{\Sigma\bar M}, \nabla^{\Sigma M}$ and $\nabla^{\Sigma N}$ the Levi-Civita connections on $\Sigma\bar M, \Sigma M$ and $\Sigma N$ respectively. For every $X, Y\in TM$, define
\begin{align*}
\bar R(X,Y)\coloneqq&[\bar\nabla_X,\bar\nabla_Y]-\bar\nabla_{[X,Y]},\\
R(X,Y)\coloneqq&[\nabla_X,\nabla_Y]-\nabla_{[X,Y]},\\
R^{\perp}(X,Y)\coloneqq&[\nabla^{\perp}_X,\nabla^{\perp}_Y]-\nabla^{\perp}_{[X,Y]},\\
R^{\Sigma\bar M}(X,Y)\coloneqq&[\nabla^{\Sigma\bar M}_X,\nabla^{\Sigma\bar M}_Y]-\nabla^{\Sigma\bar M}_{[X,Y]},\\
R^{\Sigma M}(X,Y)\coloneqq&[\nabla^{\Sigma M}_X,\nabla^{\Sigma M}_Y]-\nabla^{\Sigma M}_{[X,Y]},\\
R^{\Sigma N}(X,Y)\coloneqq&[\nabla^{\Sigma N}_X,\nabla^{\Sigma N}_Y]-\nabla^{\Sigma N}_{[X,Y]}.
\end{align*}
Denoted by $\bar\gamma, \gamma, \gamma^{\perp}$ the Clifford multiplications on  $\Sigma\bar M, \Sigma M$ and $\Sigma N$ respectively. Denoted by $\bar D, D, D^{\perp}$ the Dirac operators on  $\Sigma\bar M, \Sigma M$ and $\Sigma N$ respectively. Let $A^{\mu}$ be the shape operator of $M$ in $\bar M$ with respect to the normal vector field $\mu$, $B$ be the second fundamental form of $M$ in $\bar M$ and $H$ be the normalized mean curvature vector of $M$ in $\bar M$. If $M$ is a hypersurface of $\bar M$, we denote $A$ be the shape operator of $M$ in $\bar M$ with respect to the unit outward normal vector field.  Finally, denote $R(\iota)$ be the normalized trace of the ambient sectional curvature on the tangent space, i.e.,
\begin{align*}
R(\iota)=\dfrac{1}{m(m-1)}\sum_{i,j=1}^m\bar R(e_i,e_j,e_i,e_j),
\end{align*}
where $\set{e_i}$ is a local orthonormal frame of $TM$.

B\"ar \cite{Bar98} derived upper eigenvalue estimates for Dirac operators of closed hypersurfaces in real space forms. According to \cite{Bar98}, we know that
\begin{equation*}
\Sigma\bar M\vert_{\partial M}=
\begin{cases}
\Sigma M\otimes\Sigma N,&mn=0 \mod 2\\
\left(\Sigma M\otimes\Sigma N\right)\oplus\left(\Sigma M\otimes\Sigma N\right), &mn=1\mod2.
\end{cases}
\end{equation*}
By $D^{\Sigma N}$ we mean the Dirac operator on $M$ twisted with the bundle $\Sigma N$.

\par
A spinor $\psi$ on $\bar M$ is called a Killing spinor with Killing constant $\alpha\in\Com$ if
\begin{equation*}
\nabla_X^{\Sigma\bar M}\psi+\alpha \bar\gamma(X)\psi=0.
\end{equation*}
B\"ar \cite{Bar98} proved that if $\bar M$ admits a nontrivial Killing spinor with constant $\alpha\in\R$, then the first eigenvalue $\lambda_1\left(D^{\Sigma N}\right)$ of $D^{\Sigma N}$ (in the sense all other eigenvalue $\lambda$ of $D^{\Sigma N}$ satisfying $\abs{\lambda}\geq\abs{\lambda_1}$ ) satisfies the following estimate
\begin{equation*}
\lambda_1^2\left(D^{\Sigma N}\right)\leq m^2\abs{\alpha}^2+\dfrac{m^2}{4\vol(M)}\int_M\abs{H}^2.
\end{equation*}
 If $\alpha\in\sqrt{-1}\R$, he obtained the following estimate
\begin{equation*}
\abs{\lambda_1\left(D^{\Sigma N}\right)}\leq m\left(\abs{\alpha}+\dfrac12\norm{H}_{L^{\infty}(M)}\right).
\end{equation*}
Especially, if $\bar M$ is the Euclidean space $\R^{m+n}$, then
\begin{equation*}
\lambda_1^2(D^{\Sigma N})\leq\dfrac{m^2}{4\vol(M)}\int_M\abs{H}^2.
\end{equation*}
If $\bar M$ is the unit sphere $\field{S}^{m+n}(1)$, then
\begin{equation*}
\lambda_1^2\left(D^{\Sigma N}\right)\leq\dfrac{m^2}{4\vol(M)}\int_M\left(\abs{H}^2+1\right).
\end{equation*}
Finally, if $\bar M$ is the hyperbolic space $\Quat^{m+n}(-1)$, then
\begin{equation*}
\abs{\lambda_1\left(D^{\Sigma N}\right)}\leq \dfrac{m}{2}\left(1+\norm{H}_{L^{\infty}(M)}\right).
\end{equation*}
When $\bar M$ is the hyperbolic space $\Quat^{m+1}$, the result has been improved by Ginoux (cf. \cite{Gin03}). It was proved that
\begin{equation*}
\abs{\lambda_1\left(D^{\Sigma N}\right)}\leq \dfrac{m}{2}\left(\norm{H}_{L^{\infty}(M)}-1\right).
\end{equation*}
\par
For  hypersurface $M^m$ in $\bar M^{m+1}$, given a spinor $\psi$ on $\bar M^{m+1}$ with no zero on the hypersurface $M$,  Ginoux, Habib and Raulot introduced in \cite{GinHabRau15}  a differential operator $L_{\psi}$ acting on smooth functions on $M$ by
\begin{equation*}
L_{\psi}f\coloneqq-\Delta f-2\hin{\nabla\ln\abs{\psi}}{\nabla f}+\dfrac{m^2}{4}\left(\abs{H}^2+R(\iota)\right)f,\quad f\in C^{\infty}(M),
\end{equation*}
where $R(\iota)\coloneqq\frac{1}{m(m-1)}\left(\bar S-2\bar Ric(\nu,\nu)\right)$, $\bar S,\bar Ric$ and $\nu$ are the scalar curvature, the Ricci curvature of $\bar M$ and the unit outward norm vector field of $M$ in $\bar M$ respectively.   It was proved in \cite{GinHabRau15} that if $\bar M$ admits a nontrivial twistor-spinor $\psi$ with no zero on $M$, then
\begin{equation*}
\lambda^2_1\left(D\right)\leq\lambda_1(L_{\psi}).
\end{equation*}
Notice that if $\psi$ is a Killing spinor, then
\begin{equation*}
L_{\psi}=-\Delta +\dfrac{m^2}{4}\left(\abs{H}^2+R(\iota)\right)=-\Delta +\dfrac{m}{4(m-1)}\left(S_M+\abs{\mathring{A}}^2\right)
\end{equation*}
which is independent of $\psi$. Here  $\mathring{A}$ is the traceless part of $A$.
\par

For lower bounds estimates of submanifold Dirac operators, Hijazi and Zhang in \cite{HijZha01, HijZha01a} proved that  for $D_H\varphi=\lambda_H\varphi$, $\varphi\in \Gamma(\Sigma\bar{M})|_M$, it holds:
\begin{equation*}
\lambda_H^2\geq \frac 1 4 \sup\limits_a\inf\limits_{M_\varphi} \left(\frac{S_M+ R_{\perp,\varphi}}{1+ma^2-2a}-\frac{(m-1)m^2\abs{H}^2}{(1-ma)^2}    \right),
\end{equation*}
where $a$ is some real function on $M$,  $M_\varphi=\{x\in M|\varphi(x)\not=0\}$, and
\begin{align*}
R_{\perp,\varphi}=-\dfrac 1 2 \left(\sum\limits_{i,j,\alpha,\beta}\bar{R}_{ij\alpha\beta}e^i\cdot e^j\cdot e^\alpha\cdot e^\beta\cdot\varphi,\frac{\varphi}{|\varphi|^2}\right).
\end{align*}
Under some extra conditions on the extrinsic curvature, they also obtained some lower bound
in terms of the Yamabe constant, the curvature and volume of $M$ (see \cite{HijZha01a} for details).

\vskip6pt

\par
In this paper,   we will prove   lower and upper bound estimates for submanifold Dirac operators in terms of conformal and extrinsic  quantities. Firstly, we have the following lower bound estimate:

\begin{theorem}\label{thm:main-lower}
Let $M^m$ be a closed oriented submanifold isometrically immersed in a Riemannian spin manifold $\bar M^{n+m}$. Suppose $n=1$ or $\bar M$ is locally conformally flat. Then the  eigenvalue $\lambda$ of the Dirac operator $D^{\Sigma N}$ of the twisted bundle $\Sigma M\otimes\Sigma N$ satisfies
\begin{equation*}
\lambda^2\geq
\begin{cases}
\dfrac{4\pi(1-g_M)}{\area(M)}-\dfrac{(n-1)\int_M\abs{\mathring{A}}^2}{2\area(M)},&m=2,\\
\dfrac{m}{4(m-1)}\lambda_1(L),&m>2.
\end{cases}
\end{equation*}
Here  $\lambda_1(L)$ (if $m>2$) is the first eigenvalue of the operator $L$ defined by
\begin{equation*}
L=-\dfrac{4(m-1)}{m-2}\Delta +S_M-(n-1)\abs{\mathring{A}}^2.
\end{equation*}
Moreover, the equality implies that the Ricci curvature of $M$ satisfies
\begin{equation*}
Ric=(n-1)\sum_{\alpha=1}^n\left(\mathring{A}^{\alpha}\right)^2+\dfrac{4(m-1)\lambda^2}{m^2}g.
\end{equation*}
\end{theorem}
\begin{rem}
\begin{itemize}
\item When $m=2$,
\begin{equation*}
\int_M\abs{\mathring{A}}^2
\end{equation*}
is invariant under the conformal change of the metric $\bar g$. The equality implies that $g_M=0$  or $g_M=1$ and $\mathring{A}=0$, i.e., $M$ is a $2$-sphere or a totally umbilici $2$-torus.
\item If $m>2$, the operator is conformally invariant in the following sense. If $\bar g'=u^{4/(m-2)}\bar g$ is a metric conformal to $\bar g$, and $L'$ is similarly defined with respect to the metric $\bar g'$, then
\begin{equation*}
L'(u^{-1}f)=u^{-(m+2)/(m-2)}Lf.
\end{equation*}
\item
If $m=n=2$, then the first nonzero eigenvalue $\lambda$ of $D^{\Sigma N}$ satisfies
\begin{equation*}
\lambda^2\geq\dfrac{4\pi(1-g_M)+2\pi\abs{\chi(N)}}{\area(M)}.
\end{equation*}
\end{itemize}
\end{rem}

For a Dirac operator $D$, let $\lambda_i$ be the eigenvalues. We recall the conformal eigenvalue  $\sigma_i(D)$ of $D$ (cf. \cite{Amm03}) given by
\begin{equation*}
\sigma_i(D)=\inf_{\tilde g\in[g]}\abs{\lambda_i(\tilde g)}\vol_{M_{\tilde g}}^{1/m}.
\end{equation*}
Here $[g]$ stands for the conformal class of $g$.
Similarly, for a second positive self adjoint elliptic operator $L$, we have the conformal eigenvalue  $\lambda_i(L)$ of $L$ by
\begin{equation*}
\sigma_i(L)=\inf_{\tilde g\in[g]}\lambda_i(\tilde g)\vol_{M_{\tilde g}}^{2/m}.
\end{equation*}
Now \autoref{thm:main-lower} implies that
\begin{equation*}
\sigma_1^2\left(D^{\Sigma N}\right)\geq
\begin{cases}
4\pi(1-g_M)-\dfrac{n-1}{2}\int_M\abs{\mathring{A}}^2,&m=2,\\
\dfrac{m}{4(m-1)}\sigma_1(L),& m>2.
\end{cases}
\end{equation*}
\par

We say that $\psi$ is a twistor spinor on $\bar M$ if
\begin{equation*}
\bar\nabla_X^{\Sigma M}\psi+\dfrac{1}{m+n}\bar\gamma(X)\bar D\psi=0,\quad\forall X\in T\bar M.
\end{equation*}
By definition, we know that each Killing spinor is a twistor spinor.
For the upper bound of the Dirac operator $D^{\Sigma N}$, we will prove the following
\begin{theorem}\label{thm:main1}
Let $M, \bar M$ be  as in \autoref{thm:main-lower}. Suppose $\bar M$ admits a nontrivial twistor spinor, then there are at least $\mu$ conformal eigenvalues $\sigma_i$ of the Dirac operator $D^{\Sigma N}$ of the twisted bundle $\Sigma M\otimes\Sigma N$  such that
\begin{itemize}
\item If $m=2$,
\begin{equation*}
\sigma_i^2\leq4\pi(1-g_M)+\dfrac12\int_{M}\abs{\mathring{A}}^2.
\end{equation*}
\item If $m\geq3$,
\begin{equation*}
\sigma_i^2\leq\dfrac{m}{4(m-1)}\sigma_1\left(L_M+\abs{\mathring{A}}^2\right)=\dfrac{m}{4(m-1)}\inf_{\phi>0}\dfrac{\int_M\phi\left(L_M
+\abs{\mathring{A}}^2\right)\phi}{\left(\int_M\phi^{2m/(m-2)}\right)^{(m-2)/m}}.
\end{equation*}
Where $\mu=\dim_{\mathbb{R}}\set{\text{twistor spinors on $\bar M$}}$ and $L_M=-\tfrac{4(m-1)}{m-2}\Delta +S_M$ is the Yamabe operator of $M$.
\end{itemize}
\end{theorem}

\section{Preliminaries}
We first compare the Dirac operator on $\bar M$ with the one on $M$. We will use notations in   \cite{Bar98}. We also refer the reader to \cite{Che09,GinMor02,HijMonRol03,HijMonZha01,HijZha01,HijZha01a} and the references therein.  Basic facts concerning Clifford algebras and spinor representations can be found in classical books \cite{BerGetVer04,LawMic89}.
\par
\subsection{Algebra preliminaries}
Let $E$ be an oriented Euclidean vector space. If $\dim E=m$ is even, then the the complex Clifford algebra of $E$, denoted by $\Cl(E)$, has precisely one irreducible module, the spinor module $\Sigma E$ with dimension $2^{m/2}$. When restricted to the even subalgebra $\Cl^0(E)$ the spinor module decomposes into even and odd half-spinors $\Sigma E=\Sigma^+E\oplus\Sigma^-E$ associated the eigenspaces of the complex volume element $\omega_{\Com}=\sqrt{-1}^{m/2}\gamma_E(e_1\dotsc e_m)$.  On $\Sigma^{\pm}E$ it acts as $\pm 1$. Here $\set{e_i}$ stand for a positively oriented orthonormal frame of $E$ and $\gamma_E:\Cl(E)\To\End(E)$ stands for the Clifford multiplication.
\par
If $m$ is odd there are exactly two irreducible modules, $\Sigma^0E$ and $\Sigma^1E$, again called spinor modules. In this case $\dim\Sigma^0E=\dim\Sigma^1E=2^{(m-1)/2}$. Also the two modules $\Sigma^0E$ and $\Sigma^1E$ can be distinguished by the action of the complex volume element $\omega_{\Com}=\sqrt{-1}^{(m+1)/2}\gamma_E(e_1\cdots e_m)$. On $\Sigma^jE$ it acts as $(-1)^j$, $j=0,1$. There exists a vector space isomorphism $\Phi:\Sigma^0E\To\Sigma^1E$ such that $\Phi\circ\gamma_{E,0}=-\gamma_{E,1}\circ\Phi$, where $\gamma_{E,j}:\Cl^j(E)\To\End{\Sigma^jE}$ stand for the Clifford multiplication, $j=0,1$.
\par
Let $E$ and $F$ be two oriented Euclidean vector spaces. Let $\dim E=m$ and $\dim F=n$. We will construct the spinor module of $E\oplus F$ from those of $E$ and $F$.
\begin{itemize}
\item[Case 1.] $m$ and $n$ are both even.
\par
Put $\Sigma\coloneqq\Sigma E\otimes\Sigma F$ and define
\begin{align*}
\gamma:&E\oplus F\To\End{\Sigma},\\
\gamma(X\oplus Y)(\sigma\otimes\tau)=&\left(\gamma_E(X)\sigma\right)\otimes\tau+(-1)^{\deg\sigma}\sigma\otimes\left(\gamma_F(Y)\tau\right).
\end{align*}
Here
\begin{align*}
\deg\sigma=\begin{cases}
0,&\sigma\in\Sigma^{+}E;\\
1,&\sigma\in\Sigma^{-}E.
\end{cases}
\end{align*}
In this case
\begin{align*}
\Sigma^+\left(E\oplus F\right)=\left(\Sigma^+E\otimes\Sigma^+F\right)\oplus\left(\Sigma^-E\otimes\Sigma^-F\right),\\
\Sigma^-\left(E\oplus F\right)=\left(\Sigma^+E\otimes\Sigma^-F\right)\oplus\left(\Sigma^-E\otimes\Sigma^+F\right).
\end{align*}
\item[Case 2.] $m$ is even and $n$ is odd.
\par
Put $\Sigma^j\coloneqq\Sigma E\otimes\Sigma^jF$ for $j=0,1$. As similar to Case 1, we can define $\gamma_j:E\oplus F\To\End{\Sigma^j}$ with obvious modification.
\item[Case 3.] $m$ is odd and $n$ is even.
\par
This case is symmetric to the second one. Put $\Sigma^j\coloneqq\Sigma^jE\otimes\Sigma F$ and define
\begin{align*}
\gamma:&E\oplus F\To\End{\Sigma^j},\\
\gamma_j(X\oplus Y)(\sigma\otimes\tau)=&(-1)^{\deg\tau}(\gamma_{E,j}(X)\sigma)\otimes\tau+\sigma\otimes(\gamma_F(Y)\tau).
\end{align*}
\item[Case 4.] $m$ and $n$ are both odd.
\par
Set
\begin{align*}
\Sigma^+\coloneqq&\Sigma^0E\otimes\Sigma^0F,\\
\Sigma^-\coloneqq&\Sigma^0E\otimes\Sigma^1F,\\
\Sigma\coloneqq&\Sigma^+\oplus\Sigma^-.
\end{align*}
Recall that there exists a vector space isomorphism $\Phi:\Sigma^0F\To\Sigma^1F$ such that $\Phi\circ\gamma_{F,0}=-\gamma_{F,1}\circ\Phi$. With respect to the splitting $\Sigma=\Sigma^+\oplus\Sigma^-$, we define
\begin{align*}
\gamma:&E\oplus F\To\End{\Sigma},\\
\gamma(X\oplus Y)=&\begin{pmatrix}0&\sqrt{-1}\gamma_{E,0}(X)\otimes\Phi^{-1}+\Id\otimes(\Phi^{-1}\circ\gamma_{F,1}(Y))\\
-\sqrt{-1}\gamma_{E,0}(X)\otimes\Phi-\Id\otimes(\Phi\circ\gamma_{F,0}(Y))&0
\end{pmatrix}.
\end{align*}
\end{itemize}
\par
\subsection{Geometric preliminaries}
With respect to the orthogonal splitting $T\bar M\vert_M=TM\oplus N$, the Gauss formula says
\begin{equation*}
\bar\nabla_X=\begin{pmatrix}\nabla_X&-B(X,\cdot)^*\\
B(X,\cdot)&\nabla_X^{\perp}
\end{pmatrix}.
\end{equation*}
The following equations are well known, i.e., Gauss equations, Codazzi equations and Ricci equations (cf. \cite{Xin03}). For all $X, Y, Z\in TM, \mu\in N$,
\begin{align*}
\bar R(X,Y)Z=&R(X,Y)Z+\hin{B(X,Z)}{B(Y,\cdot)}-\hin{B(Y,Z)}{B(X,\cdot)}+(\nabla_XB)(Y,Z)-(\nabla_YB)(X,Z),\\
\bar R(X,Y)\mu=&(\nabla_YA)^{\mu}(X)-(\nabla_XA)^{\mu}(Y)+R^{\perp}(X,Y)\mu+\hin{B(A^{\mu}(X),Y)}{\cdot}-\hin{B(A^{\mu}(Y),X)}{\cdot}.
\end{align*}

\par
Using a standard formula (cf. \cite{LawMic89}), we have
\begin{align*}
\nabla_X^{\Sigma\bar M\vert_M}=&\nabla_X^{\Sigma M}\otimes\Id+\Id\otimes\nabla_X^{\Sigma N}+\dfrac12\sum_{\alpha=1}^n\bar\gamma(A^{\alpha}(X)\cdot\nu_{\alpha}),\\
R^{\Sigma\bar M\vert_M}(X,Y)=&R^{\Sigma M}(X,Y)\otimes\Id+\Id\otimes R^{\Sigma N}(X,Y)+\dfrac14\sum_{\alpha=1}^n\gamma([A^{\alpha}(X),A^{\alpha}(Y)])\otimes\Id\\
&+\dfrac14\sum_{\alpha,\beta=1}^n\left(\hin{A^{\alpha}(X)}{A^{\beta}(Y)}-\hin{A^{\alpha}(Y)}{A^{\beta}(X)}\right)\Id\otimes\gamma^{\perp}(\nu_{\alpha}\cdot\nu_{\beta})\\
&+\dfrac12\sum_{\alpha=1}^n\bar\gamma\left(((\nabla_XA)^{\alpha}(Y)-(\nabla_YA)^{\alpha}(X))\cdot\nu_{\alpha}\right).
\end{align*}
Here $\set{\nu_{\alpha}}$ is a local orthonormal frame of the normal bundle $N$.
\par
 Define
\begin{equation*}
\tilde D\coloneqq\sum_{i=1}^m\bar\gamma(e_i)\nabla_{e_i}^{\Sigma M\otimes\Sigma N}.
\end{equation*}
Then (cf. \cite{Bar98})
\begin{equation*}
\tilde D^2=
\begin{cases}
\left(D^{\Sigma N}\right)^2,&mn=0 \mod 2;\\
\left(D^{\Sigma N}\oplus(-D^{\Sigma N})\right)^2,& mn=1 \mod2.
\end{cases}
\end{equation*}

\par
Recall the Bochner formula (cf., \cite{Jos17, LawMic89}),
\begin{equation*}
\left(D^{\Sigma N}\right)^2=\left(\nabla^{\Sigma M\otimes\Sigma N}\right)^*\nabla^{\Sigma M\otimes\Sigma N}+\mathcal{R}^{\Sigma N},
\end{equation*}
where
\begin{equation*}
\mathcal{R}^{\Sigma N}=\dfrac12\bar\gamma(e_i\cdot e_j)R^{\Sigma M\otimes\Sigma N}(e_i,e_j).
\end{equation*}

Recall the curvature decomposition of $\bar R$. Denoted $\bar P$ by the Schouten tensor which is defined by
\begin{align*}
\bar P_{AB}\coloneqq\dfrac{1}{n+m-2}\left(\bar Ric_{AB}-\dfrac{\bar S}{2(n+m-1)}\bar g_{AB}\right),\quad 1\leq A, B\leq n+m,
\end{align*}
the Weyl tensor $\bar W$ is given by
\begin{align*}
\bar W_{ABCD}\coloneqq\bar R_{ABCD}-\left(\bar P_{AC}\bar g_{BD}+\bar P_{BD}\bar g_{AC}-\bar P_{AD}\bar g_{BC}-\bar P_{BC}\bar g_{AD}\right).
\end{align*}
Therefore, for every orthonormal $4$-frame $\set{e_A,e_B,e_C,e_D}$, we have
\begin{align*}
\bar W_{ABCD}=\bar R_{ABCD}.
\end{align*}

\begin{lem}\label{lem:curvature}
\begin{equation}\label{eq:0}
\mathcal{R}^{\Sigma N}=\dfrac{m(m-1)}{4}\left(R(\iota)+\abs{H}^2\right)+\dfrac14\sum_{i=1}^m\left(\sum_{\alpha=1}^n\bar\gamma\left(\mathring{A}^{\alpha}(e_i)\cdot\nu_{\alpha}\right)\right)^2-\dfrac18\bar W_{ij\alpha\beta}\bar\gamma(e_i\cdot e_j\cdot\nu_{\alpha}\cdot\nu_{\beta}).
\end{equation}
\end{lem}
\begin{proof}A standard computation (cf. \cite{LawMic89}) gives a formula
\begin{equation}\label{eq:1}
\mathcal{R}^{\Sigma N}=\dfrac18\hin{R(e_i,e_j)e_k}{e_l}\bar\gamma(e_i\cdot e_j\cdot e_k\cdot e_l)+\dfrac18\hin{R^{\perp}(e_i,e_j)\nu_{\alpha}}{\nu_{\beta}}\bar\gamma(e_i\cdot e_j\cdot\nu_{\alpha}\cdot\nu_{\beta}).
\end{equation}
The first term is
\begin{equation}\label{eq:2}
\dfrac{S_M}{4}=\dfrac14\left(\sum_{i,j=1}^m\bar R(e_i,e_j,e_i,e_j)+m(m-1)\abs{H}^2-\abs{\mathring{A}}^2\right).
\end{equation}
According to the Codazzi equation, we compute the second term as follows,
\begin{align*}
&\dfrac18\hin{R^{\perp}(e_i,e_j)\nu_{\alpha}}{\nu_{\beta}}\bar\gamma(e_i\cdot e_j\cdot\nu_{\alpha}\cdot\nu_{\beta})\\
=&\dfrac18\left(\hin{\bar R(e_i,e_j)\nu_{\alpha}}{\nu_{\beta}}+\hin{A^{\alpha}(e_j)}{A^{\beta}(e_i)}-\hin{A^{\alpha}(e_i)}{A^{\beta}(e_j)}\right)\bar\gamma(e_i\cdot e_j\cdot\nu_{\alpha}\cdot\nu_{\beta})\\
=&\dfrac18\hin{\bar W(e_i,e_j)\nu_{\alpha}}{\nu_{\beta}}\bar\gamma(e_i\cdot e_j\cdot\nu_{\alpha}\cdot\nu_{\beta})+\dfrac18\left(\hin{\mathring{A}^{\alpha}(e_j)}{\mathring{A}^{\beta}(e_i)}-\hin{\mathring{A}^{\alpha}(e_i)}{\mathring{A}^{\beta}(e_j)}\right)\bar\gamma(e_i\cdot e_j\cdot\nu_{\alpha}\cdot\nu_{\beta})\\
=&\dfrac14\left(\sum_{i=1}^m\sum_{\alpha,\beta=1}^n\bar\gamma\left(\mathring{A}^{\alpha}(e_i)\cdot\nu_{\alpha}\cdot \mathring{A}^{\beta}(e_i)\cdot\nu_{\beta}\right)+\abs{\mathring{A}}^2\right)+\dfrac18\hin{\bar W(e_i,e_j)\nu_{\alpha}}{\nu_{\beta}}\bar\gamma(e_i\cdot e_j\cdot\nu_{\alpha}\cdot\nu_{\beta}),
\end{align*}
where we used the fact
\begin{align*}
\bar W_{ij\alpha\beta}=\bar R_{ij\alpha\beta},\quad\forall i\neq j, \alpha\neq\beta.
\end{align*}
Thus, the second term is
\begin{equation}\label{eq:3}
\dfrac14\left(\sum_{i=1}^m\sum_{\alpha,\beta=1}^n\bar\gamma\left(\mathring{A}^{\alpha}(e_i)\cdot\nu_{\alpha}\cdot \mathring{A}^{\beta}(e_i)\cdot\nu_{\beta}\right)+\abs{\mathring{A}}^2\right)-\dfrac18\bar W_{ij\alpha\beta}\bar\gamma(e_i\cdot e_j\cdot\nu_{\alpha}\cdot\nu_{\beta}).
\end{equation}

Now \eqref{eq:0} follows from \eqref{eq:1}, \eqref{eq:2} and \eqref{eq:3}.
\end{proof}
\begin{rem}\label{rem:22}
\begin{enumerate}
\item If $n=1$,
\begin{equation*}
\mathcal{R}^{\Sigma N}=\dfrac14S_M=\dfrac{m(m-1)}{4}\left(R(\iota)+\abs{H}^2\right)-\dfrac14\abs{\mathring{A}}^2.
\end{equation*}
\item If $m=2, n=2$,
\begin{align*}
\mathcal{R}^{\Sigma N}\vert_{\Sigma^{\pm}}=&\dfrac12\kappa_M\pm\dfrac12\kappa_N=\dfrac{1}{2}\left(\bar R(e_1,e_2,e_1,e_2)+\abs{H}^2\right)-\dfrac14\abs{\mathring{A}}^2\pm\dfrac12\kappa_N,\\
-\dfrac14\sum_{i=1}^m\left(\sum_{\alpha=1}^n\bar\gamma\left(\mathring{A}^{\alpha}(e_i)\cdot\nu_{\alpha}\right)\right)^2\vert_{\Sigma^{\pm}}=&\dfrac14\abs{\mathring{A}}^2\mp\dfrac12\left(\kappa_N-\bar R(e_1,e_2,\nu_1,\nu_2)\right).
\end{align*}
\end{enumerate}
Here
\begin{equation*}
\kappa_N=\hin{R^{\perp}(e_1,e_2)\nu_2}{\nu_1}.
\end{equation*}
A direct consequence is
\begin{equation*}
\int_M\abs{\mathring{A}}^2\geq2\abs{2\pi\chi(N)-\int_M\bar R(e_1,e_2,\nu_1,\nu_2)}.
\end{equation*}
Therefore,
\begin{equation*}
\chi(M)+\abs{\chi(N)-\dfrac{1}{2\pi}\int_M\bar R(e_1,e_2,\nu_1,\nu_2)}\leq\dfrac{1}{2\pi}\left(\int_M\bar R(e_1,e_2,e_1,e_2)+\abs{H}^2\right).
\end{equation*}
In particular, if $\bar M$ is flat and $M$ is minimal (cf. \cite{Iri02}), then
\begin{equation*}
\chi(M)+\abs{\chi(N)}\leq0.
\end{equation*}
\end{rem}
\begin{proof}The first remark is obvious. For the first part of the second remark, we refer the reader to H. Iriyeh's paper \cite{Iri02}. For the second part, we have
\begin{align*}
&-\dfrac14\sum_{i=1}^m\left(\sum_{\alpha=1}^n\bar\gamma\left(\mathring{A}^{\alpha}(e_i)\cdot\nu_{\alpha}\right)\right)^2\vert_{\Sigma^{\pm}}\\
=&\dfrac14\sum_{i=1}^2\sum_{\alpha,\beta=1}^2\bar\gamma\left(\mathring{A}^{\alpha}(e_i)\cdot\mathring{A}^{\beta}(e_i)\cdot\nu_{\alpha}\cdot\nu_{\beta}\right)\\
=&\dfrac14\abs{\mathring{A}}^2+\dfrac14\sum_{i=1}^2\sum_{\alpha\neq\beta}\bar\gamma\left(\mathring{A}^{\alpha}(e_i)\cdot\mathring{A}^{\beta}(e_i)\cdot\nu_{\alpha}\cdot\nu_{\beta}\right)\\
=&\dfrac14\abs{\mathring{A}}^2+\dfrac14\sum_{i=1}^2\sum_{\alpha\neq\beta}\bar\gamma\left(A^{\alpha}(e_i)\cdot A^{\beta}(e_i)\cdot\nu_{\alpha}\cdot\nu_{\beta}\right)\\
=&\dfrac14\abs{\mathring{A}}^2+\dfrac14\sum_{i=1}^2\sum_{j\neq k}\sum_{\alpha\neq\beta}\hin{A^{\alpha}(e_i)}{e_j}\hin{A^{\beta}(e_i)}{e_k}\bar\gamma\left(e_j\cdot e_k\cdot\nu_{\alpha}\cdot\nu_{\beta}\right)\\
=&\dfrac14\abs{\mathring{A}}^2+\dfrac12\sum_{i=1}^2\left(\hin{A^{1}(e_i)}{e_1}\hin{A^{2}(e_i)}{e_2}-\hin{A^{1}(e_i)}{e_2}\hin{A^{2}(e_i)}{e_1}\right)\bar\gamma\left(e_1\cdot e_2\cdot\nu_{1}\cdot\nu_{2}\right)\\
=&\dfrac14\abs{\mathring{A}}^2\pm\dfrac12\left(\kappa_N-\bar R(e_1,e_2,\nu_1,\nu_2)\right).
\end{align*}
The third part follows from the fact
\begin{equation*}
-\dfrac14\sum_{i=1}^m\left(\sum_{\alpha=1}^n\bar\gamma\left(\mathring{A}^{\alpha}(e_i)\cdot\nu_{\alpha}\right)\right)^2\geq0.
\end{equation*}
Hence,
\begin{equation*}
\dfrac12\int_M\abs{\mathring{A}}^2\geq\int_M\abs{\kappa_N-\bar R(e_1,e_2,\nu_1,\nu_2)}\geq\abs{2\pi\chi(N)-\int_M\bar R(e_1,e_2,\nu_1,\nu_2)}.
\end{equation*}
Finally, according to the Gauss equation,
\begin{equation*}
\kappa_M=\bar R(e_1,e_2,e_1,e_2)+\abs{H}^2-\dfrac12\abs{\mathring{A}}^2,
\end{equation*}
we obtain
\begin{equation*}
\chi(M)+\abs{\chi(N)-\dfrac{1}{2\pi}\int_M\bar R(e_1,e_2,\nu_1,\nu_2)}\leq\dfrac{1}{2\pi}\left(\int_M\bar R(e_1,e_2,e_1,e_2)+\abs{H}^2\right).
\end{equation*}
\end{proof}

\subsection{Conformal transformation}
Consider a conformal change $\bar g'=e^{2u}\bar g$ of $\bar M$, then there is an isometric between $\Sigma\bar M$ and $\Sigma\bar M'$, $\psi\mapsto\psi'$, with
\begin{equation*}
\bar\nabla_X'\psi'=\left(\bar\nabla_X\psi-\dfrac12\bar\gamma(X\cdot\bar\nabla u)\psi-\dfrac12X(u)\psi\right)'.
\end{equation*}
Moreover, when restricted to the boundary, we have
\begin{align}
\nabla_X'\sigma'=&\left(\nabla_X\sigma-\dfrac12\gamma(X\cdot\nabla u)\sigma-\dfrac12X(u)\sigma\right)',\label{eq:conformal1}\\
\nabla_X^{'\perp}\tau'=&\left(\nabla_X^{\perp}\tau\right)'.\label{eq:conformal2}
\end{align}
Here $\sigma\in\Gamma(\Sigma M),\tau\in\Gamma(\Sigma N)$. The first equation follows from the fact that the metric restricted to boundary also satisfies $g'=e^{2u}g$. The second equation can be proved as follows. According to the Gauss formula,
\begin{equation*}
\hin{\nabla_X^{\perp}\nu_{\alpha}}{\nu_{\beta}}=\hin{\bar\nabla_X\nu_{\alpha}}{\nu_{\beta}},\quad\forall X\in TM.
\end{equation*}
In particular,
\begin{equation*}
\omega^{\perp}_{\alpha\beta}(X)=\bar\omega_{\alpha\beta}(X).
\end{equation*}
Here $\omega^{\perp}$ and $\bar\omega$ are the connection $1$-forms on the normal bundle $N$ and the target manifold $\bar M$ respectively. Since we have the transformation formula between connection $1$-forms
\begin{equation*}
\bar\omega_{AB}'(X)=\bar\omega_{AB}(X)+e_A(u)\hin{X}{e_B}-e_B(u)\hin{X}{e_A}.
\end{equation*}
Hence,
\begin{equation*}
\omega^{'\perp}_{\alpha\beta}(X)=\bar\omega'_{\alpha\beta}(X)=\bar\omega_{\alpha\beta}(X)=\omega_{\alpha\beta}^{'\perp}(X).
\end{equation*}
Now according to definition of the connection on $N$, we get (cf. \cite{LawMic89})
\begin{equation*}
\nabla_X^{'\perp}\tau'=\left(\nabla_X\tau\right)'.
\end{equation*}
\par
Now we can prove the following
\begin{lem}The Dirac operator on the twisted bundle $\Sigma M\otimes\Sigma N$ is conformal invariant, i.e., for every $\psi\in\Gamma\left(\Sigma M\otimes\Sigma N\right)$
\begin{equation*}
D^{\Sigma' N}\left(e^{-(m-1)u/2}\psi'\right)=e^{-(m+1)u/2}\left(D^{\Sigma N}\psi\right)'.
\end{equation*}
\end{lem}
\begin{proof} Without loss generality, set $\psi=\sigma\otimes\tau$, then according to \eqref{eq:conformal1}, we have
\begin{align*}
D'\left(e^{-(m-1)/2}\sigma'\right)=&e^{-(m+1)u/2}\left(D\sigma\right)'.
\end{align*}
Hence by using \eqref{eq:conformal2}, we get
\begin{align*}
D^{\Sigma'N}\left(e^{-(m-1)u/2}(\sigma\otimes\tau)'\right)=&D^{\Sigma'N}\left(e^{-(m-1)u/2}\sigma'\otimes\tau'\right)\\
=&D^{\Sigma'N}\left(e^{-(m-1)u/2}\sigma'\right)\otimes\tau'+\gamma'(e_i')e^{-(m-1)u/2}\sigma'\otimes\nabla^{'\perp}_{e_i'}\tau'\\
=&e^{-(m+1)u/2}\left(D\sigma\otimes\tau+\gamma(e_i)\sigma\otimes\nabla^{\perp}_{e_i}\tau\right)\\
=&e^{-(m+1)u/2}\left(D^{\Sigma N}(\sigma\otimes\tau)\right)'.
\end{align*}
\end{proof}

\section{Lower bound estimate}
In this section, we will give a conformal lower bound of the first eigenvalue of the Dirac operator on the twisted bundle $\Sigma M\otimes\Sigma N$, i.e., we will give a proof of \autoref{thm:main-lower}.

\begin{proof}[Proof of \autoref{thm:main-lower}]For every smooth function $f$, we have the following weighted Bochner formula (cf. \cite{CheJosSun17})
\begin{equation}\label{eq:basic}
\begin{split}
\dfrac{m-1}{m}\int_M\exp\left(f\right)\abs{D^{\Sigma N}\psi}^2=&\int_{M}\exp(f)\left(\dfrac{m-1}{2}\Delta f-\dfrac{(m-1)(m-2)}{4}\abs{\nabla f}^2+\mathcal{R}^{\Sigma N}_{\psi}\right)\abs{\psi}^2\\
&+\int_M\exp((1-m)f)\abs{P^{\Sigma N}\left(\exp\left(\dfrac{m}{2}f\right)\psi\right)}^2,
\end{split}
\end{equation}
where
\begin{equation*}
\mathcal{R}^{\Sigma N}_{\psi}\abs{\psi}^2=\rin{\mathcal{R}^{\Sigma N}\psi}{\psi},
\end{equation*}
and $P^{\Sigma N}$ is the twistor operator defined by
\begin{equation*}
P^{\Sigma N}_X\psi\coloneqq\nabla_X^{\Sigma M\otimes\Sigma N}+\dfrac1m\bar\gamma(X)D^{\Sigma N}\psi.
\end{equation*}
\par
First, we estimate the curvature term $\mathcal{R}^{\Sigma N}_{\psi}$. According to the proof of \autoref{lem:curvature},  \eqref{eq:3} implies that
\begin{align*}
&\dfrac18\hin{R^{\perp}(e_i,e_j)\nu_{\alpha}}{\nu_{\beta}}\rin{\bar\gamma(e_i\cdot e_j\cdot\nu_{\alpha}\cdot\nu_{\beta})\psi}{\psi}+\dfrac18\bar W_{ij\alpha\beta}\bar\gamma(e_i\cdot e_j\cdot\nu_{\alpha}\cdot\nu_{\beta})\\
=&\dfrac14\left(\sum_{i,\alpha}\abs{\bar\gamma(A^{\alpha}(e_i))\psi}^2-\sum_i\abs{\sum_{\alpha}\bar\gamma(A^{\alpha}(e_i)\cdot\nu_{\alpha})\psi}^2\right)\\
=&\dfrac14\left(\sum_{i,\alpha}\abs{\bar\gamma(\mathring{A}^{\alpha}(e_i))\psi}^2-\sum_i\abs{\sum_{\alpha}\bar\gamma(\mathring{A}^{\alpha}(e_i)\cdot\nu_{\alpha})\psi}^2\right)\\
=&\dfrac14\left(n\sum_{i,\beta}\abs{\bar\gamma(\mathring{A}^{\beta}(e_i))\psi-\dfrac1n\sum_{\alpha}\bar\gamma(\nu_{\beta}\cdot\mathring{A}^{\alpha}(e_i)\cdot\nu_{\alpha})\psi}^2-(n-1)\abs{\mathring{A}}^2\abs{\psi}^2\right).
\end{align*}
In particular,
\begin{equation}\label{eq:3.1}
\dfrac18\hin{R^{\perp}(e_i,e_j)\nu_{\alpha}}{\nu_{\beta}}\rin{\bar\gamma(e_i\cdot e_j\cdot\nu_{\alpha}\cdot\nu_{\beta})\psi}{\psi}\geq-\dfrac18\bar W_{ij\alpha\beta}\bar\gamma(e_i\cdot e_j\cdot\nu_{\alpha}\cdot\nu_{\beta})-\dfrac{n-1}{4}\abs{\mathring{A}}^2\abs{\psi}^2.
\end{equation}

Insert \eqref{eq:3.1} into \eqref{eq:1} to get
\begin{equation}\label{eq:curvature}
\mathcal{R}^{\Sigma N}_{\psi}\geq\dfrac{S_M-(n-1)\abs{\mathring{A}}^2}{4}-\dfrac{\bar W_{ij\alpha\beta}\rin{\bar\gamma(e_i\cdot e_j\cdot\nu_{\alpha}\cdot\nu_{\beta})\psi}{\psi}}{8\abs{\psi}^2}.
\end{equation}

\par
Suppose $\psi$ is an eigenspinor of $D^{\Sigma N}$ associated with $\lambda$, i.e.,
\begin{equation*}
D^{\Sigma N}\psi=\lambda\psi.
\end{equation*}
Inserting \eqref{eq:curvature} into \eqref{eq:basic}, we obtain
\begin{equation}\label{eq:basic1}
\begin{split}
\dfrac{m-1}{m}\lambda^2\int_Me^f\abs{\psi}^2\geq&\int_{M}e^f\left(\dfrac{m-1}{2}\Delta f-\dfrac{(m-1)(m-2)}{4}\abs{\nabla f}^2+\dfrac{S_M-(n-1)\abs{\mathring{A}^2}}{4}\right)\abs{\psi}^2.
\end{split}
\end{equation}
\par
We consider two cases.
\begin{enumerate}[{Case} 1.]
\item $m=2$.
\par
In this case, we choose $f$ as a solution of the following PDE
\begin{equation*}
\Delta f+\kappa_M-\dfrac{n-1}{2}\mathring{A}^2=\dfrac{4\pi(1-g_M)}{\area(M)}-\dfrac{(n-1)\int_M\abs{\mathring{A}}^2}{2\area(M)},\quad\int_Mf=0,
\end{equation*}
on $M$. Therefore, according to \eqref{eq:basic1}, we get
\begin{equation*}
\lambda^2\geq\dfrac{4\pi(1-g_M)}{\area(M)}-\dfrac{(n-1)\int_M\abs{\mathring{A}}^2}{2\area(M)}.
\end{equation*}
Moreover, if $n=2$, according to \autoref{rem:22}
\begin{equation*}
\mathcal{R}^{\Sigma N}_{\psi}\vert_{\Sigma^{\pm}}=\mathcal{R}^{\Sigma N}_{\psi}\vert_{\Sigma^{\pm}}=\dfrac12\kappa_M\pm\dfrac12\kappa_N.
\end{equation*}
A direct computation implies that if $D^{\Sigma N}\psi=\lambda\psi$, then $D^{\Sigma N}\psi^{\pm}=\lambda\psi^{\mp}$. Since $\lambda\neq0$, we get $\psi^{\pm}\neq0$ since $\psi$ is a nontrivial eigenspinor. Using a similar argument mentioned before, one can proved that
\begin{equation*}
\lambda^2\geq\dfrac{4\pi(1-g_M)\pm2\pi\chi(N)}{\area(M)}.
\end{equation*}
Therefore,
\begin{equation*}
\lambda^2\geq\dfrac{4\pi(1-g_M)+2\pi\abs{\chi(N)}}{\area(M)}.
\end{equation*}
Here we used two formulae
\begin{equation*}
\int_M\kappa_M=2\pi\chi(M)=4\pi(1-g_M),
\end{equation*}
and
\begin{equation*}
\int_M\kappa_N=2\pi\chi(N).
\end{equation*}
\item $m>2$.
\par
In this case, \eqref{eq:basic1} implies that for every positive function $u$,
\begin{align}\label{eq:bochner}
\dfrac{m-1}{m}\lambda^2\int_{M}u^{1-m/(m-2)}\abs{\psi}^2\geq\int_{M}u^{-m/(m-2)}\left(-\dfrac{m-1}{m-2}\Delta u++\dfrac{S_M-(n-1)\abs{\mathring{A}^2}}{4}u\right)\abs{\psi}^2.
\end{align}
Choose $u$ as an eigenfunction of the operator $L$, i.e.,
\begin{equation*}
Lu=-\dfrac{4(m-1)}{m-2}\Delta u+\left(S_M-(n-1)\abs{\mathring{A}}^2\right)u=\lambda_1(L)u.
\end{equation*}
Moreover, we can choose $u$ satisfying
\begin{equation*}
\int_Mu^2=\vol(M).
\end{equation*}
Then the inequality \eqref{eq:bochner} implies that
\begin{equation*}
\lambda^2\geq\dfrac{m}{4(m-1)}\lambda_1(L).
\end{equation*}

\end{enumerate}

\par
Next, we will consider the limit case. If suppose
\begin{equation*}
\lambda^2=\dfrac{4\pi(1-g_M)}{\area(M)}-\dfrac{(n-1)\int_M\abs{\mathring{A}}^2}{2\area(M)}.
\end{equation*}
as $m=2$ is the case and
\begin{equation*}
\lambda^2=\dfrac{m}{4(m-1)}\lambda_1(L).
\end{equation*}
as $m>2$ is the case. Consider a new metric $\bar g'=e^{-2f}\bar g$, then $\tilde\psi=e^{(m-1)f/2}\psi'$ ($f=\tfrac{2\log u}{2-m}$ if $m>2$) satisfies
\begin{equation}\label{eq:limit1}
\nabla_{e_i}^{\Sigma' M\otimes\Sigma' N}\tilde\psi+\dfrac{\lambda e^f}{m}\bar\gamma'(e_i')\tilde\psi=0.
\end{equation}
Consequently, $\abs{\tilde\psi}_{g'}\neq0$ is a constant on $M$.
Moreover, the equality in \eqref{eq:3.1} gives
\begin{equation}\label{eq:limit2}
\bar\gamma(\mathring{A}^{\alpha}(e_i)\cdot\nu_{\alpha})\psi=\bar\gamma(\mathring{A}^{\beta}(e_i)\cdot\nu_{\beta})\psi,\quad\forall i, \alpha,\beta.
\end{equation}
Form \eqref{eq:limit1}, we get
\begin{equation*}
\sum_{i=1}^m\bar\gamma'(e_i')R^{\Sigma' M\otimes\Sigma' N}(e_i',e_j')\tilde\psi=\dfrac{2(m-1)\lambda^2e^{2f}}{m^2}\bar\gamma'(e_j')\tilde\psi-\dfrac{\lambda e^f}{m}\bar\gamma'(\nabla'f\cdot e_j')\tilde\psi-\lambda e^fe_j'(f)\tilde\psi.
\end{equation*}
Thus,
\begin{equation*}
\dfrac{(1-m)}{m}\lambda e^fe_j'(f)\abs{\tilde\psi}^2_{g'}=0.
\end{equation*}
Therefore, $f$ is a constant and $f=0$ according to the normalizing condition. As a consequence,
\begin{equation*}
\sum_{i=1}^m\bar\gamma(e_i)R^{\Sigma M\otimes\Sigma N}(e_i,e_j)\psi=\dfrac{2(m-1)\lambda^2}{m^2}\bar\gamma(e_j)\psi.
\end{equation*}
On the other hand, one can get (cf. \cite{Hij86,LawMic89}),
\begin{align*}
\sum_{i=1}^m\bar\gamma(e_i)R^{\Sigma M\otimes\Sigma N}(e_i,e_j)\psi=&\dfrac14\sum_{i,k,l=1}^m\hin{R(e_i,e_j)e_k}{e_l}\bar\gamma(e_i\cdot e_k\cdot e_l)\psi\\
&+\dfrac14\sum_{i=1}^m\sum_{\alpha,\beta=1}^n\hin{R^{\perp}(e_i,e_j)\nu_{\alpha}}{\nu_{\beta}}\bar\gamma(e_i\cdot\nu_{\alpha}\cdot\nu_{\beta})\psi\\
=&\dfrac12\bar\gamma\left(Ric(e_j)\right)\psi\\
&-\dfrac14\sum_{i=1}^m\sum_{\alpha=1}^n\bar\gamma\left(\mathring{B}(e_j,e_i)\cdot\mathring{A}^{\alpha}(e_i)\cdot\nu_{\alpha}+\mathring{A}^{\alpha}(e_i)\cdot\nu_{\alpha}\cdot\mathring{B}(e_j,e_i)\right)\psi.
\end{align*}
According to \eqref{eq:limit2}, we get
\begin{align*}
\sum_{i=1}^m\bar\gamma(e_i)R^{\Sigma M\otimes\Sigma N}(e_i,e_j)\psi=&\dfrac12\bar\gamma\left(Ric(e_j)\right)\psi+\dfrac{1-n}{2}\sum_{\alpha=1}^n\bar\gamma\left(\left(\mathring{A}^{\alpha}\right)^2(e_j)\right)\psi.
\end{align*}
Summarize these identities, we get
\begin{equation}\label{eq:ricci}
\dfrac12\bar\gamma\left(Ric(e_j)\right)\psi+\dfrac{1-n}{2}\sum_{\alpha=1}^n\bar\gamma\left(\left(\mathring{A}^{\alpha}\right)^2(e_j)\right)\psi=\dfrac{2(m-1)\lambda^2}{m^2}\bar\gamma(e_j)\psi.
\end{equation}
Since $\psi$ can not vanish anywhere on $M$, then \eqref{eq:ricci} implies that
\begin{equation*}
Ric=(n-1)\sum_{\alpha=1}^n\left(\mathring{A}^{\alpha}\right)^2+\dfrac{4(m-1)\lambda^2}{m^2}g.
\end{equation*}
\end{proof}

\section{Upper bound estimate}
In this section, we want to bound the first conformal eigenvalue of the Dirac operator $D^{\Sigma N}$ by extrinsic data provided $\bar M$ admits a twistor spinor $\psi$, i.e.,
\begin{equation*}
\bar P_X\psi\coloneqq\bar\nabla_X^{\Sigma M}\psi+\dfrac{1}{m+n}\bar\gamma(X)\bar D\psi=0,\quad\forall X\in T\bar M.
\end{equation*}
When restricted to the boundary, we first prove the following Lemma.
\begin{lem}\label{lem:twistor}
For every tangent vector field  $X\in\Gamma(TM)$, we have
\begin{equation*}
P_X^{\Sigma N}\psi=\bar P_X\psi+\dfrac1m\sum_{i=1}^m\bar\gamma(e_i)\bar P_{e_i}\psi-\dfrac12\sum_{\alpha=1}^n\bar\gamma\left(\mathring{A}^{\alpha}(X)\cdot\nu_{\alpha}\right)\psi.
\end{equation*}
\end{lem}
\begin{proof}According to the definition of the connections given in the previous sections, we get
\begin{align*}
\tilde D\psi=\sum_{i=1}^m\bar\gamma(e_i)\bar P_{e_i}\psi+\dfrac{m}{m+n}\bar D\psi+\dfrac m2\bar\gamma(H)\psi.
\end{align*}
Thus (cf. \cite{Bar98}),
\begin{align*}
P^{\Sigma N}_X\psi\coloneqq&\nabla_X^{\Sigma M\otimes\Sigma N}\psi+\dfrac{1}{m}\bar\gamma(X)D^{\Sigma N}\psi\\
=&\nabla_X^{\Sigma M\otimes\Sigma N}\psi+\dfrac{1}{m}\bar\gamma(X)\tilde D\psi\\
=&\nabla_X^{\Sigma\bar M\vert_M}\psi-\dfrac12\sum_{\alpha=1}^n\bar\gamma(A^{\alpha}(X)\cdot\nu_{\alpha})\psi+\dfrac1m\sum_{i=1}^m\bar\gamma(e_i)\bar P_{e_i}\psi+\dfrac{1}{m+n}\bar D\psi+\dfrac{1}{2}\bar\gamma(H)\psi\\
=&\bar P_X\psi+\dfrac1m\sum_{i=1}^m\bar\gamma(e_i)\bar P_{e_i}\psi-\dfrac12\sum_{\alpha=1}^n\bar\gamma\left(\mathring{A}^{\alpha}(X)\cdot\nu_{\alpha}\right)\psi.
\end{align*}
\end{proof}

Now, we give a proof of \autoref{thm:main1}
\begin{proof}[Proof of \autoref{thm:main1}]Applying the weighted Bochner formula \eqref{eq:basic}, (replacing $\psi$ by $f\psi$ and $f$ by $u$),
\begin{equation*}
\begin{split}
\dfrac{m-1}{m}\int_Me^u\abs{D^{\Sigma N}(f\psi)}^2=&\int_{M}e^u\left(\dfrac{m-1}{2}\Delta u-\dfrac{(m-1)(m-2)}{4}\abs{\nabla u}^2+\mathcal{R}^{\Sigma N}_{\psi}\right)\abs{f\psi}^2\\
&+\int_Me^{(1-m)u}\abs{P^{\Sigma N}\left(e^{mu/2}f\psi\right)}^2.
\end{split}
\end{equation*}
In particular, taking $fe^{mu/2}=1$, we get
\begin{equation}\label{eq:upper}
\begin{split}
\dfrac{m-1}{m}\int_Me^u\abs{D^{\Sigma N}(e^{-mu/2}\psi)}^2=&\int_{M}e^{(1-m)u}\left(\dfrac{m-1}{2}\Delta u-\dfrac{(m-1)(m-2)}{4}\abs{\nabla u}^2+\mathcal{R}^{\Sigma N}_{\psi}\right)\abs{\psi}^2\\
&+\int_Me^{(1-m)u}\abs{P^{\Sigma N}\psi}^2.
\end{split}
\end{equation}
Now \autoref{lem:curvature} gives
\begin{align*}
\mathcal{R}^{\Sigma N}_{\psi}\abs{\psi}^2=&\dfrac{m(m-1)}{4}\left(R(\iota)+\abs{H}^2\right)\abs{\psi}^2-\dfrac14\sum_{i=1}^m\abs{\sum_{\alpha=1}^n\bar\gamma\left(\mathring{A}^{\alpha}(e_i)\cdot\nu_{\alpha}\right)\psi}^2\\
&-\dfrac{1}{8}\bar W_{ij\alpha\beta}\rin{\bar\gamma(e_i\cdot e_j\cdot\nu_{\alpha}\cdot\nu_{\beta})\psi}{\psi},
\end{align*}
and \autoref{lem:twistor} gives
\begin{equation*}
\abs{P^{\Sigma N}\psi}^2=\dfrac14\sum_{i=1}^m\abs{\sum_{\alpha=1}^n\bar\gamma\left(\mathring{A}^{\alpha}(e_i)\cdot\nu_{\alpha}\right)\psi}^2,
\end{equation*}
if $\psi$ is a twistor spinor of $\Sigma\bar M$.  Therefore, \eqref{eq:upper} can be rewritten as follows:
\begin{align*}
&\dfrac{m-1}{m}\int_Me^u\abs{D^{\Sigma N}(e^{-mu/2}\psi)}^2\\
=&\int_{M}e^{(1-m)u}\left(\dfrac{m-1}{2}\Delta u-\dfrac{(m-1)(m-2)}{4}\abs{\nabla u}^2+\dfrac{m(m-1)}{4}\left(R(\iota)+\abs{H}^2\right)\right)\abs{\psi}^2\\
&-\int_M\dfrac18e^{(1-m)u}\bar W_{ij\alpha\beta}\rin{\bar\gamma(e_i\cdot e_j\cdot\nu_{\alpha}\cdot\nu_{\beta})\psi}{\psi}.
\end{align*}
\par
Since $\psi$ is a nontrivial twistor spinor on $\bar M$, we know that the zeros of $\psi$ is isolated (\cite{Fri90}). In particular, $\psi$ is nontrivial on $M$. Considering a conformal change of the metric $\bar g'=e^{-2u}\bar g$, we get
\begin{align*}
&\dfrac{m-1}{m}\dfrac{\int_{M'}\abs{D^{\Sigma' N}(e^{-u/2}\psi')}^2_{g'}}{\int_{M'}\abs{e^{-u/2}\psi'}_{g'}^2}\\
=&\dfrac{\int_{M}e^{(1-m)u}\left(\dfrac{m-1}{2}\Delta u-\dfrac{(m-1)(m-2)}{4}\abs{\nabla u}^2+\dfrac{m(m-1)}{4}\left(R(\iota)+\abs{H}^2\right)\right)\abs{\psi}^2}{\int_Me^{-(1+m)u}\abs{\psi}^2}\\
&-\dfrac{\int_M\dfrac18e^{(1-m)u}\bar W_{ij\alpha\beta}\rin{\bar\gamma(e_i\cdot e_j\cdot\nu_{\alpha}\cdot\nu_{\beta})\psi}{\psi}}{\int_Me^{-(1+m)u}\abs{\psi}^2}.
\end{align*}

By assumption, $n=1$ or $\bar M$ is locally conformally flat, we obtain that the second term of the above equation is zero.
We consider two case
\begin{enumerate}[U1]
\item $m=2$. We get
\begin{align*}
\dfrac{\int_{M'}\abs{D^{\Sigma' N}(e^{-u/2}\psi')}^2_{g'}}{\int_{M'}\abs{e^{-u/2}\psi'}_{g'}^2}=\dfrac{\int_{M}e^{-u}\left(\Delta u+\left(R(\iota)+\abs{H}^2\right)\right)\abs{\psi}^2}{\int_Me^{-3u}\abs{\psi}^2}.
\end{align*}
\par
We consider the following Liouville-type equations
\begin{equation*}
\Delta u_j+\kappa_g+\dfrac12\abs{\mathring{A}}^2+\varepsilon_j=\mu_j e^{-2u_j},\quad\int_Me^{-2u_j}=1.
\end{equation*}
Here  $\set{\varepsilon_j}$ is some sequence consists of positive numbers such that $\lim_{j\to\infty}\varepsilon_j=0$ and $\mu_j$ is constant for each $j$. For the existence of $\varepsilon_j$, we refer the reader to Chen-Lin's paper \cite{CheLin03} for genus $g_M\geq1$  and  Djadli's paper \cite{Dja08} for arbitrary  genus. Then
 \begin{equation*}
 \lim_{j\to\infty}\mu_j=4\pi(1-g_M)+\dfrac12\int_M\abs{\mathring{A}}^2.
 \end{equation*}
Thus the first conformal eigenvalue of $D^{\Sigma N}$ satisfies
\begin{equation*}
\sigma^2_i=\inf\lambda_i^2\area(M)\leq \lim_{j\to\infty}\mu_j=4\pi(1-g_M)+\dfrac12\int_M\abs{\mathring{A}}^2.
\end{equation*}
\item $m>2.$
\par
In this case, let $e^u=\phi^{2/(2-m)}$, where $\phi$ is a positive function. Then a direct computation implies that
\begin{align*}
\dfrac{\int_{M'}\abs{D^{\Sigma' N}(e^{-u/2}\psi')}^2_{g'}}{\int_{M'}\abs{e^{-u/2}\psi'}_{g'}^2}=\dfrac{\int_M\left(-\dfrac{m}{m-2}\Delta\phi+\dfrac{m^2}{4}\left(\abs{H}^2+R(\iota)\right)\phi\right)\phi^{m/(m-2)}\abs{\psi}^2}{\int_M\phi^{2(m+1)/(m-2)}\abs{\psi}^2}.
\end{align*}
We consider the following nonlinear equations
\begin{equation*}
-\dfrac{4(m-1)}{m-2}\Delta\phi_j+m(m-1)\left(\abs{H}^2+R(\iota)\right)\phi_j=\tau_j\phi^{p_j-1},
\end{equation*}
or equivalently
\begin{equation*}
\left(L_M+\abs{\mathring{A}}^2\right)\phi_j=-\dfrac{4(m-1)}{m-2}\Delta\phi_j+\left(S_M+\abs{\mathring{A}}^2\right)\phi_j=\tau_j\phi^{p_j-1},
\end{equation*}
where
\begin{equation*}
\tau_j=\inf_{\phi>0}\dfrac{\int_M\phi\left(L_M+\abs{\mathring{A}}^2\right)\phi}{\left(\int_M\phi^{p_j}\right)^{1/(2p_j)}},
\end{equation*}
and $2<p_j<2m/(m-2)$. It is obvious that $\tau_j\geq0$.
\par
Choose $\phi_j>0$ satisfying
\begin{equation*}
\left(L_M+\abs{\mathring{A}}^2\right)\phi_j=\tau_j\phi_j^{p_j-1},\quad\int_M\phi_j^{p_j}=1.
\end{equation*}
By using a similar argument to the Yamabe constant (cf. \cite{LeePar87}), it can be shown that $\tau_j\leq\sigma_1\left(L_M+\abs{\mathring{A}}^2\right)$ and
\begin{equation*}
\lim_{p_j\to 2m/(m-2)}\tau_j=\sigma_1\left(L_M+\abs{\mathring{A}}^2\right)=\inf_{\phi>0}\dfrac{\int_M\phi\left(L_M+\abs{\mathring{A}}^2\right)\phi}{\left(\int_M\phi^{2m/(m-2)}\right)^{(m-2)/m}}.
\end{equation*}
Thus, we obtain
\begin{equation*}
\sigma^2_i=\inf\lambda_i^2\vol^{2/m}\leq \dfrac{m}{4(m-1)}\lim_{p_j\to 2m/(m-2)}\tau_j=\dfrac{m}{4(m-1)}\sigma_1\left(L_M+\abs{\mathring{A}}^2\right).
\end{equation*}
\end{enumerate}
\end{proof}


\begin{thebibliography}{25}
\expandafter\ifx\csname natexlab\endcsname\relax\def\natexlab#1{#1}\fi
\providecommand{\url}[1]{\texttt{#1}}
\providecommand{\href}[2]{#2}
\providecommand{\path}[1]{#1}
\providecommand{\DOIprefix}{doi:}
\providecommand{\ArXivprefix}{arXiv:}
\providecommand{\URLprefix}{URL: }
\providecommand{\Pubmedprefix}{pmid:}
\providecommand{\doi}[1]{\href{http://dx.doi.org/#1}{\path{#1}}}
\providecommand{\Pubmed}[1]{\href{pmid:#1}{\path{#1}}}
\providecommand{\bibinfo}[2]{#2}
\ifx\xfnm\relax \def\xfnm[#1]{\unskip,\space#1}\fi
\bibitem[{Ammann(2003)}]{Amm03}
\bibinfo{author}{Ammann, B.}, \bibinfo{year}{2003}.
\newblock \bibinfo{title}{A variational problem in conformal {S}pin geometry}.
\newblock \bibinfo{journal}{Habilitation (Universit\"at Hamburg, 2003)}
  \URLprefix
  \url{https://www.uni-regensburg.de/Fakultaeten/nat_Fak_I/ammann/preprints/habil/habil.ps}.
\bibitem[{B\"{a}r(1992)}]{Bar92}
\bibinfo{author}{B\"{a}r, C.}, \bibinfo{year}{1992}.
\newblock \bibinfo{title}{Lower eigenvalue estimates for {D}irac operators}.
\newblock \bibinfo{journal}{Math. Ann.} \bibinfo{volume}{293},
  \bibinfo{pages}{39--46}.
\newblock \URLprefix \url{https://doi.org/10.1007/BF01444701},
  \DOIprefix\doi{10.1007/BF01444701}.
\bibitem[{B\"{a}r(1998)}]{Bar98}
\bibinfo{author}{B\"{a}r, C.}, \bibinfo{year}{1998}.
\newblock \bibinfo{title}{Extrinsic bounds for eigenvalues of the {D}irac
  operator}.
\newblock \bibinfo{journal}{Ann. Global Anal. Geom.} \bibinfo{volume}{16},
  \bibinfo{pages}{573--596}.
\newblock \URLprefix \url{https://doi.org/10.1023/A:1006550532236},
  \DOIprefix\doi{10.1023/A:1006550532236}.
\bibitem[{Berline et~al.(2004)Berline, Getzler and Vergne}]{BerGetVer04}
\bibinfo{author}{Berline, N.}, \bibinfo{author}{Getzler, E.},
  \bibinfo{author}{Vergne, M.}, \bibinfo{year}{2004}.
\newblock \bibinfo{title}{Heat kernels and {D}irac operators}.
\newblock Grundlehren Text Editions, \bibinfo{publisher}{Springer-Verlag,
  Berlin}.
\newblock \bibinfo{note}{Corrected reprint of the 1992 original}.
\bibitem[{Chen and Lin(2003)}]{CheLin03}
\bibinfo{author}{Chen, C.}, \bibinfo{author}{Lin, C.}, \bibinfo{year}{2003}.
\newblock \bibinfo{title}{Topological degree for a mean field equation on
  {R}iemann surfaces}.
\newblock \bibinfo{journal}{Comm. Pure Appl. Math.} \bibinfo{volume}{56},
  \bibinfo{pages}{1667--1727}.
\newblock \URLprefix \url{https://doi.org/10.1002/cpa.10107},
  \DOIprefix\doi{10.1002/cpa.10107}.
\bibitem[{Chen(2009)}]{Che09}
\bibinfo{author}{Chen, D.}, \bibinfo{year}{2009}.
\newblock \bibinfo{title}{Extrinsic estimates for eigenvalues of the {D}irac
  operator}.
\newblock \bibinfo{journal}{Math. Z.} \bibinfo{volume}{262},
  \bibinfo{pages}{349--361}.
\newblock \URLprefix \url{https://doi.org/10.1007/s00209-008-0376-8},
  \DOIprefix\doi{10.1007/s00209-008-0376-8}.
\bibitem[{Chen et~al.(2017)Chen, Jost, Sun and Zhu}]{CheJosSun17}
\bibinfo{author}{Chen, Q.}, \bibinfo{author}{Jost, J.}, \bibinfo{author}{Sun,
  L.}, \bibinfo{author}{Zhu, M.}, \bibinfo{year}{2017}.
\newblock \bibinfo{title}{Estimates for solutions of {D}irac equations and an
  application to a geometric elliptic-parabolic problem}.
\newblock \bibinfo{journal}{ArXiv e-prints}
  \href{http://arxiv.org/abs/1707.03151}{{\tt arXiv:1707.03151}}.
\bibitem[{Djadli(2008)}]{Dja08}
\bibinfo{author}{Djadli, Z.}, \bibinfo{year}{2008}.
\newblock \bibinfo{title}{Existence result for the mean field problem on
  {R}iemann surfaces of all genuses}.
\newblock \bibinfo{journal}{Commun. Contemp. Math.} \bibinfo{volume}{10},
  \bibinfo{pages}{205--220}.
\newblock \URLprefix \url{https://doi.org/10.1142/S0219199708002776},
  \DOIprefix\doi{10.1142/S0219199708002776}.
\bibitem[{Friedrich(1980)}]{Fri80}
\bibinfo{author}{Friedrich, T.}, \bibinfo{year}{1980}.
\newblock \bibinfo{title}{Der erste {E}igenwert des {D}irac-{O}perators einer
  kompakten, {R}iemannschen {M}annigfaltigkeit nichtnegativer
  {S}kalarkr\"{u}mmung}.
\newblock \bibinfo{journal}{Math. Nachr.} \bibinfo{volume}{97},
  \bibinfo{pages}{117--146}.
\newblock \URLprefix \url{https://doi.org/10.1002/mana.19800970111},
  \DOIprefix\doi{10.1002/mana.19800970111}.
\bibitem[{Friedrich(1990)}]{Fri90}
\bibinfo{author}{Friedrich, T.}, \bibinfo{year}{1990}.
\newblock \bibinfo{title}{On the conformal relation between twistors and
  {K}illing spinors}, in: \bibinfo{booktitle}{Proceedings of the {W}inter
  {S}chool on {G}eometry and {P}hysics ({S}rn\'{i}, 1989)}, pp.
  \bibinfo{pages}{59--75}.
\bibitem[{Friedrich(2000)}]{Fri00}
\bibinfo{author}{Friedrich, T.}, \bibinfo{year}{2000}.
\newblock \bibinfo{title}{Dirac operators in {R}iemannian geometry}.
  volume~\bibinfo{volume}{25} of \textit{\bibinfo{series}{Graduate Studies in
  Mathematics}}.
\newblock \bibinfo{publisher}{American Mathematical Society, Providence, RI}.
\newblock \URLprefix \url{https://doi.org/10.1090/gsm/025},
  \DOIprefix\doi{10.1090/gsm/025}. \bibinfo{note}{translated from the 1997
  German original by Andreas Nestke}.
\bibitem[{Ginoux(2003)}]{Gin03}
\bibinfo{author}{Ginoux, N.}, \bibinfo{year}{2003}.
\newblock \bibinfo{title}{Une nouvelle estimation extrins\`eque du spectre de
  l'op\'{e}rateur de {D}irac}.
\newblock \bibinfo{journal}{C. R. Math. Acad. Sci. Paris}
  \bibinfo{volume}{336}, \bibinfo{pages}{829--832}.
\newblock \URLprefix \url{https://doi.org/10.1016/S1631-073X(03)00206-1},
  \DOIprefix\doi{10.1016/S1631-073X(03)00206-1}.
\bibitem[{Ginoux(2009)}]{Gin09}
\bibinfo{author}{Ginoux, N.}, \bibinfo{year}{2009}.
\newblock \bibinfo{title}{The {D}irac spectrum}. volume \bibinfo{volume}{1976}
  of \textit{\bibinfo{series}{Lecture Notes in Mathematics}}.
\newblock \bibinfo{publisher}{Springer-Verlag, Berlin}.
\newblock \URLprefix \url{https://doi.org/10.1007/978-3-642-01570-0},
  \DOIprefix\doi{10.1007/978-3-642-01570-0}.
\bibitem[{Ginoux et~al.(2015)Ginoux, Habib and Raulot}]{GinHabRau15}
\bibinfo{author}{Ginoux, N.}, \bibinfo{author}{Habib, G.},
  \bibinfo{author}{Raulot, S.}, \bibinfo{year}{2015}.
\newblock \bibinfo{title}{A new upper bound for the {D}irac operators on
  hypersurfaces}.
\newblock \bibinfo{journal}{Pacific J. Math.} \bibinfo{volume}{278},
  \bibinfo{pages}{79--101}.
\newblock \URLprefix \url{https://doi.org/10.2140/pjm.2015.278.79},
  \DOIprefix\doi{10.2140/pjm.2015.278.79}.
\bibitem[{Ginoux and Morel(2002)}]{GinMor02}
\bibinfo{author}{Ginoux, N.}, \bibinfo{author}{Morel, B.},
  \bibinfo{year}{2002}.
\newblock \bibinfo{title}{On eigenvalue estimates for the submanifold {D}irac
  operator}.
\newblock \bibinfo{journal}{Internat. J. Math.} \bibinfo{volume}{13},
  \bibinfo{pages}{533--548}.
\newblock \URLprefix \url{https://doi.org/10.1142/S0129167X0200140X},
  \DOIprefix\doi{10.1142/S0129167X0200140X}.
\bibitem[{Hijazi(1986)}]{Hij86}
\bibinfo{author}{Hijazi, O.}, \bibinfo{year}{1986}.
\newblock \bibinfo{title}{A conformal lower bound for the smallest eigenvalue
  of the {D}irac operator and {K}illing spinors}.
\newblock \bibinfo{journal}{Comm. Math. Phys.} \bibinfo{volume}{104},
  \bibinfo{pages}{151--162}.
\newblock \URLprefix \url{http://projecteuclid.org/euclid.cmp/1104114937}.
\bibitem[{Hijazi et~al.(2003)Hijazi, Montiel and Rold\'{a}n}]{HijMonRol03}
\bibinfo{author}{Hijazi, O.}, \bibinfo{author}{Montiel, S.},
  \bibinfo{author}{Rold\'{a}n, A.}, \bibinfo{year}{2003}.
\newblock \bibinfo{title}{Dirac operators on hypersurfaces of manifolds with
  negative scalar curvature}.
\newblock \bibinfo{journal}{Ann. Global Anal. Geom.} \bibinfo{volume}{23},
  \bibinfo{pages}{247--264}.
\newblock \URLprefix \url{https://doi.org/10.1023/A:1022808916165},
  \DOIprefix\doi{10.1023/A:1022808916165}.
\bibitem[{Hijazi et~al.(2001)Hijazi, Montiel and Zhang}]{HijMonZha01}
\bibinfo{author}{Hijazi, O.}, \bibinfo{author}{Montiel, S.},
  \bibinfo{author}{Zhang, X.}, \bibinfo{year}{2001}.
\newblock \bibinfo{title}{Dirac operator on embedded hypersurfaces}.
\newblock \bibinfo{journal}{Math. Res. Lett.} \bibinfo{volume}{8},
  \bibinfo{pages}{195--208}.
\newblock \URLprefix \url{https://doi.org/10.4310/MRL.2001.v8.n2.a8},
  \DOIprefix\doi{10.4310/MRL.2001.v8.n2.a8}.
\bibitem[{Hijazi and Zhang(2001a)}]{HijZha01}
\bibinfo{author}{Hijazi, O.}, \bibinfo{author}{Zhang, X.},
  \bibinfo{year}{2001}a.
\newblock \bibinfo{title}{Lower bounds for the eigenvalues of the {D}irac
  operator. {I}. {T}he hypersurface {D}irac operator}.
\newblock \bibinfo{journal}{Ann. Global Anal. Geom.} \bibinfo{volume}{19},
  \bibinfo{pages}{355--376}.
\newblock \URLprefix \url{https://doi.org/10.1023/A:1010749808691},
  \DOIprefix\doi{10.1023/A:1010749808691}.
\bibitem[{Hijazi and Zhang(2001b)}]{HijZha01a}
\bibinfo{author}{Hijazi, O.}, \bibinfo{author}{Zhang, X.},
  \bibinfo{year}{2001}b.
\newblock \bibinfo{title}{Lower bounds for the eigenvalues of the {D}irac
  operator. {II}. {T}he submanifold {D}irac operator}.
\newblock \bibinfo{journal}{Ann. Global Anal. Geom.} \bibinfo{volume}{20},
  \bibinfo{pages}{163--181}.
\newblock \URLprefix \url{https://doi.org/10.1023/A:1011663603699},
  \DOIprefix\doi{10.1023/A:1011663603699}.
\bibitem[{Iriyeh(2002)}]{Iri02}
\bibinfo{author}{Iriyeh, H.}, \bibinfo{year}{2002}.
\newblock \bibinfo{title}{Minimal submanifolds in {R}iemannian {S}pin manifolds
  with parallel spinor fields}.
\newblock \bibinfo{journal}{J. Geom. Phys.} \bibinfo{volume}{41},
  \bibinfo{pages}{258--273}.
\newblock \URLprefix \url{https://doi.org/10.1016/S0393-0440(01)00059-6},
  \DOIprefix\doi{10.1016/S0393-0440(01)00059-6}.
\bibitem[{Jost(2017)}]{Jos17}
\bibinfo{author}{Jost, J.}, \bibinfo{year}{2017}.
\newblock \bibinfo{title}{Riemannian geometry and geometric analysis}.
\newblock Universitext. \bibinfo{edition}{seventh} ed.,
  \bibinfo{publisher}{Springer, Cham}.
\newblock \URLprefix \url{https://doi.org/10.1007/978-3-319-61860-9},
  \DOIprefix\doi{10.1007/978-3-319-61860-9}.
\bibitem[{Lawson and Michelsohn(1989)}]{LawMic89}
\bibinfo{author}{Lawson, H.}, \bibinfo{author}{Michelsohn, M.},
  \bibinfo{year}{1989}.
\newblock \bibinfo{title}{Spin geometry}. volume~\bibinfo{volume}{38} of
  \textit{\bibinfo{series}{Princeton Mathematical Series}}.
\newblock \bibinfo{publisher}{Princeton University Press, Princeton, NJ}.
\bibitem[{Lee and Parker(1987)}]{LeePar87}
\bibinfo{author}{Lee, J.}, \bibinfo{author}{Parker, T.}, \bibinfo{year}{1987}.
\newblock \bibinfo{title}{The {Y}amabe problem}.
\newblock \bibinfo{journal}{Bull. Amer. Math. Soc. (N.S.)}
  \bibinfo{volume}{17}, \bibinfo{pages}{37--91}.
\newblock \URLprefix \url{https://doi.org/10.1090/S0273-0979-1987-15514-5},
  \DOIprefix\doi{10.1090/S0273-0979-1987-15514-5}.
\bibitem[{Xin(2003)}]{Xin03}
\bibinfo{author}{Xin, Y.}, \bibinfo{year}{2003}.
\newblock \bibinfo{title}{Minimal submanifolds and related topics}.
  volume~\bibinfo{volume}{8} of \textit{\bibinfo{series}{Nankai Tracts in
  Mathematics}}.
\newblock \bibinfo{publisher}{World Scientific Publishing Co., Inc., River
  Edge, NJ}.
\newblock \URLprefix \url{https://doi.org/10.1142/9789812564382},
  \DOIprefix\doi{10.1142/9789812564382}.

\end{thebibliography}

\end{document}